%% file: PAMS_main.tex
\documentclass[11pt,twoside,reqno]{amsart}
\usepackage{import}
\include{preamble.tex}

\title{Symmetric functions and the explicit moment problem for abelian groups}

\author{Roger Van Peski}
\email{\textcolor{blue}{\href{rogervanpeski@gmail.com}{rogervanpeski@gmail.com}}}
\address{Department of Mathematics, Columbia University,USA}

\date{\today}

\begin{document}

\maketitle

\begin{abstract}
Recently, Sawin and Wood \cite{sawin2022moment} proved a formula for the distribution of a random abelian group $G$ in terms of its $H$-moments $\mathbb{E} [\#\operatorname{Sur}(G,H)]$. We show that properties of Macdonald polynomials yield an alternate proof.
\end{abstract}

\tableofcontents

\section{Introduction}

The purpose of this note is to point out a link between two directions of recent work on random groups: the moment method on the one hand, and combinatorial techniques coming from Macdonald symmetric functions on the other.

Random or pseudorandom groups appear across number theory, combinatorics, and topology, see e.g. \cite{wood2023probability,sawin2022moment} and the references therein. In any of these settings, one typically has a sequence of random groups $(G_n)_{n \geq 1}$, and wishes to show convergence in distribution to a universal limit random group $G$. The moment method is a technique for showing this convergence by showing that for every fixed group\footnote{with appropriate modifiers such as finite, abelian, etc.} $H$, the so-called \emph{$H$-moments} $\E[\#\Sur(G_n,H)]$ converge to their limiting values $\E[\#\Sur(G,H)]$. Under suitable growth hypotheses on the moments $\E[\#\Sur(G,H)]$, they uniquely determine the group $G$, and the above moment convergence implies convergence in distribution $G_n \to G$. This was first shown by Wood \cite{wood2017distribution}, and has since become a standard tool to study random groups, see for instance Wood \cite{wood2015random,wood2018cohen}, Nguyen-Wood \cite{nguyen2022random,nguyen2022local}, Nguyen and the author \cite{nguyen2022universality}, M\'esz\'aros \cite{meszaros2020distribution,meszaros2023cohen}, Cheong-Huang \cite{cheong2021cohen} -Kaplan \cite{cheong2022generalizations} and -Yu \cite{cheong2023cokernel}, and Lee \cite{lee2022universality,lee2023joint} for more recent works using the moment method.

The problem of determining whether a random group exists with given moments, and whether it is unique, is often referred to as the moment problem for random groups (again with suitable modifiers such as finite, abelian, etc.). Sufficient conditions were given in the above-mentioned work, but the moment method still requires that one has already identified the candidate limiting group $G$ and computed its moments. Given a nice enough list of numbers $M_H$ indexed by finite abelian groups $H$, results of \cite{wood2017distribution} guaranteed that there is a unique random group $G$ with moments $\E[\#\Sur(G,H)]=M_H$ for every $H$, but it was not at all clear how to find the distribution of $G$, i.e. find probabilities such as $\Pr(G \cong \Z/9\Z)$. We refer to this latter problem as the \emph{explicit moment problem}. The solution to the explicit moment problem was treated by Sawin-Wood \cite{sawin2022finite} in a related context, and shortly after was systematized in a general framework of so-called diamond categories in \cite{sawin2022moment}. The latter framework applies to finite abelian groups as well as many other contexts, and gives the probability that a random object takes a given value as a certain nontrivial infinite linear combination of its moments.

For simplicity consider random abelian $p$-groups $G$ (i.e. $\#G$ is a power of $p$). Any fixed abelian $p$-group $H$ is isomorphic to $\bigoplus_{i \geq 1} \Z/p^{\la_i}\Z$ for some integers $\la_1 \geq \la_2 \geq \ldots \geq 0$, where only finitely many $\la_i$ are nonzero. Such nonnegative, eventually-zero sequences $\la = (\la_1,\la_2,\ldots)$ are called integer partitions, and we denote the set of them by $\Y$. Given a partition $\la \in \Y$, define the conjugate partition $\la' = (\la_1',\la_2',\ldots)$ by $\la_i':= \#\{j: \la_j \geq i\}$, see \Cref{fig:ferrers}.

\begin{figure}[H]
\centering
\begin{minipage}[b]{30ex}
    \centering
    \[\begin{ytableau} *(white) & *(white) & *(white) & *(white) & *(white)  \\ *(white) & *(white) \\ *(white) & *(white) \\ *(white)  \end{ytableau} \]
    \caption*{$\la = (5,2,2,1)$}
    \label{fig:minipage1}
  \end{minipage}
  \begin{minipage}[b]{0.45\textwidth}
    \centering
 \begin{ytableau} *(white) & *(white) & *(white) & *(white)  \\ *(white) & *(white) & *(white) \\ *(white) \\ *(white) \\ *(white)  \end{ytableau}
    \caption*{$\la' = (4,3,1,1,1)$}
    \label{fig:minipage2}
  \end{minipage}
\caption{A partition $\la \in \Y$ may be drawn as a Ferrers diagram with row lengths corresponding to parts $\la_i$. The diagram of the conjugate partition $\la'$ defined above is then given by reflecting this one.} \label{fig:ferrers}
\end{figure}
~Additionally, define
\begin{align}
\begin{split}\label{eq:intro_notation}
G_\la &:= \bigoplus_i \Z/p^{\la_i}\Z \quad \quad \quad \quad \text{(the prime $p$ is implicit in this notation)} \\ 
|\la| &:= \sum_i \la_i \\ 
n(\la) &:= \sum_{i \geq 1} (i-1)\la_i = \sum_{i \geq 1} \binom{\la_i'}{2}.
\end{split}
\end{align}
Then the distribution of a random abelian $p$-group $G$ is given in terms of its moments as follows, where we use $q$-Pochhammer notation
\begin{equation}\label{eq:def_qp}
(a;q)_n := \prod_{i=1}^n (1-aq^{i-1}).
\end{equation}

\begin{thm}\label{thm:invert_moments_no_level}
Fix $p$ prime and let $G$ be a random finite abelian $p$-group with finite $H$-moments $M_H := \E[\#\Sur(G,H)]$ for each finite abelian $p$-group $H$. Then for any $\nu \in \Y$,
\begin{equation}\label{eq:invert_moments_no_level}
\Pr(G \cong G_\nu) = \sum_{\substack{\mu \in \Y: \\ \mu_1' \geq \nu_1' \geq \mu_2' \geq \nu_2' \geq \ldots}} \frac{(-1)^{|\mu|-|\nu|} p^{-n(\nu)-n(\mu)-|\mu|}}{\prod_{i \geq 1} (p^{-1};p^{-1})_{\mu_i'-\nu_i'}(p^{-1};p^{-1})_{\nu_i'-\mu_{i+1}'}} M_{G_\mu},
\end{equation}
provided that the sum on the right hand side converges absolutely. Furthermore, the law of $G$ is the unique probability measure with moments $M_H$.
\end{thm}

\Cref{thm:invert_moments_no_level} applies with the same proof to modules over any complete discrete valuation ring with finite residue field, such as $\F_q\llbracket T \rrbracket$, see \Cref{rmk:dvr}; we state it for abelian $p$-groups ($\Z_p$-modules) here for concreteness only. One can also easily extend \Cref{thm:invert_moments_no_level} to groups with order divisible by multiple primes by including a product over primes in \eqref{eq:invert_moments_no_level}, see \Cref{thm:multiple_primes} for a precise statement. A moment inversion formula of the form \eqref{eq:invert_moments_no_level} was shown previously in \cite[Proposition 6.2]{sawin2022moment}\footnote{We note that \cite{sawin2022moment} proves more about this formula: they show that given a collection of real numbers $M_H$ (not \emph{a priori} given by moments) with \eqref{eq:invert_moments_no_level} converging absolutely, then there exists a measure with those moments if and only if the right hand side of \eqref{eq:invert_moments_no_level} is always nonnegative. They also show that convergence of moments to $M_H$ suffices to conclude convergence in distribution to $G$ (`robustness'). We do not try to reprove these facts here, as our interest is in the connection between the combinatorics of \eqref{eq:invert_moments_no_level} and symmetric functions.}, and our purpose here is to show that it also follows directly from properties of Macdonald symmetric functions. 

The Macdonald polynomials $P_\la(x_1,\ldots,x_n;q,t)$ are certain symmetric polynomials in $n$ variables $x_1,\ldots,x_n$, indexed by integer partitions $\la = (\la_1 \geq \la_2 \geq \ldots \geq 0)$, and featuring two extra parameters $q,t$. When $q=t$ they become the well-known Schur polynomials. At general $q,t$ their combinatorics has provided tools to analyze eigenvalues of random matrices and various exactly solvable models in statistical mechanics, beginning with work of Borodin-Corwin \cite{borodin2014macdonald}, see also Borodin-Petrov \cite{borodin2014integrable} and the references therein. 

The case $q=0$, known as the \emph{Hall-Littlewood polynomials}, encodes the combinatorics of abelian $p$-groups when $t=1/p$ through the so-called Hall algebra, see \cite[Chapters II and III]{mac} and \cite[Section 6]{nguyen2022universality}. The relation between combinatorics of abelian $p$-groups and Hall-Littlewood polynomials is related to the latter's role as spherical functions on $p$-adic groups, see \cite[Chapter V]{mac}, and Hall-Littlewood polynomials also appeared in random matrix theory over finite fields in work of Fulman \cite{fulman_main,fulman2016hall}, and over the $p$-adic numbers in Delaunay-Jouhet \cite{delaunay2014p} and Fulman-Kaplan \cite{fulman2018random}.

A series of works \cite{van2020limits,vanpeski2021halllittlewood,nguyen2022universality,van2023local} used the Hall-Littlewood case to study random abelian $p$-groups arising in random matrix theory over $\Z_p$ and $\Z$. After the computations for the most recent of these works \cite{van2023local}, we realized that a certain key symmetric function lemma, \cite[Lemma 4.2]{van2023local}, actually yields a special case of the moment inversion formula \Cref{thm:invert_moments_no_level} for a certain specific measure on abelian $p$-groups (though it was not phrased in this way in \cite{van2023local}), and the proof generalizes to other measures as well. We emphasize that \cite[Lemma 4.2]{van2023local} appeared in that work as a computational tool with no group-theoretic interpretation, and it was a surprise to us that it led to a symmetric function interpretation and proof of the moment inversion formula for abelian $p$-groups. 

In a sense our proof of \Cref{thm:invert_moments_no_level} and the proof of \cite[Proposition 6.2]{sawin2022moment} have the same combinatorial core. One has
\begin{equation}
\E[\#\Sur(G,G_\mu)] = \sum_{\substack{\nu \in \Y: \nu_i \geq \mu_i \text{ for all }i}} \Pr(G \cong G_\nu) \#\Sur(G_\nu,G_\mu)
\end{equation}
where the condition $\nu_i \geq \mu_i \text{ for all }i$ comes because $\#\Sur(G_\nu,G_\mu)$ is zero otherwise. In other words, the moments $\E[\#\Sur(G,G_\mu)]$ are upper-triangular linear combinations of the probabilities $\Pr(G \cong G_\nu)$, and the problem becomes how to invert this system, i.e. how to take the right (infinite) linear combination of moments such that all but one of the $\Pr(G \cong G_\nu)$ terms cancel. In \cite{sawin2022moment} this combinatorics is understood using M\"obius functions, and here we show that it is also encoded in the ring of symmetric functions. 

The techniques here and in \cite{sawin2022moment} both apply in the case of finite modules over a complete discrete valuation ring with finite residue field. This is their only intersection of which we are currently aware, but interestingly, both ways of understanding \Cref{thm:invert_moments_no_level} generalize in different directions. The combinatorics of \cite{sawin2022moment} is proven in the much more general diamond category setting, while the symmetric function combinatorics underlying our proof holds in the Macdonald setting before specializing to the Hall-Littlewood one, see \Cref{thm:specs_cancel}. We note that Macdonald polynomials, at other special values of the parameters $q$ and $t$, appear in a variety of contexts; it would be interesting to see whether the combinatorial skeleton here also specializes to nontrivial results in any of these.

\textbf{Acknowledgments.} We thank Alexei Borodin for helpful conversations, encouragement to write this note, and comments on it, Hoi H. Nguyen for useful comments and questions, Will Sawin for several helpful remarks on an earlier version, and Melanie Matchett Wood for discussions regarding the explicit moment problem and for pointing us to \cite{sawin2022finite} (and later \cite{sawin2022moment}) after we found the proof here. This work was partially supported by an NSF Graduate Research Fellowship under grant \#1745302, and was written up while supported by the European Research Council (ERC), Grant Agreement No. 101002013.

\section{Macdonald symmetric functions}\label{sec:macdonald_background}

In this section we introduce standard facts about symmetric functions from \cite{mac} and specializations of symmetric functions from \cite{borodin2014macdonald}. 

\subsection{Partitions and symmetric polynomials.} We denote by $\cP$ the set of all integer partitions $(\la_1,\la_2,\ldots)$, i.e. sequences of nonnegative integers $\la_1 \geq \la_2 \geq \cdots$ which are eventually $0$. We call the integers $\la_i$ the \emph{parts} of $\la$, set $\la_i' = \#\{j: \la_j \geq i\}$, and write $m_i(\la) = \#\{j: \la_j = i\} = \la_i'-\la_{i+1}'$. We write $\len(\la)$ for the number of nonzero parts, and denote the set of partitions of length $\leq n$ by $\cP_n$. We write $\mu \prec \la$ or $\la \succ \mu$ if $\la_1 \geq \mu_1 \geq \la_2 \geq \mu_2 \geq \cdots$, and refer to this condition as \emph{interlacing}. A stronger partial order is defined by containment of Ferrers diagrams (see \Cref{fig:ferrers}), which we write as $\mu \subset \la$, meaning $\mu_i \leq \la_i$ for all $i$. Finally, we denote the partition with all parts equal to zero by $\emptyset$. 

We denote by $\La_n$ the ring $\C[x_1,\ldots,x_n]^{S_n}$ of symmetric polynomials in $n$ variables $x_1,\ldots,x_n$. For a symmetric polynomial $f$, we will often write $f(\bx)$ for $f(x_1,\ldots,x_n)$ when the number of variables is clear from context. We will also use the shorthand $\bx^\la:= x_1^{\la_1} x_2^{\la_2} \cdots x_n^{\la_n}$ for $\la \in \cP_n$. A simple $\C$-basis for $\La_n$ is given by the \emph{monomial symmetric polynomials} $\{m_\la(\bx): \la \in \Y_n\}$ defined by 
\[
m_\la(\bx) = \sum_{\sigma \in S_n} \sigma(\bx^\la)
\]
where $\sigma$ acts by permuting the variables. It is also a very classical fact that the power sum symmetric polynomials 
\[p_k(\bx) = \sum_{i=1}^n x_i^k, k =1,\ldots,n\]
are algebraically independent and algebraically generate $\La_n$, and so by defining 
\begin{equation*}
    p_\la(\bx) := \prod_{i \geq 1} p_{\la_i}(\bx)
\end{equation*}
for $\la \in \Y$ with $\la_1 \leq n$, we have that $\{p_\la(\bx): \la_1 \leq n\}$ forms another basis for $\La_n$.

\subsection{Symmetric functions.} It is often convenient to consider symmetric polynomials in an arbitrarily large or infinite number of variables, which we formalize as follow, heavily borrowing from the introductory material in \cite{van2022q}. One has a chain of maps
\[
\cdots \to \La_{n+1} \to \La_n \to \La_{n-1} \to \cdots \to 0
\]
where the map $\La_{n+1} \to \La_n$ is given by setting $x_{n+1}$ to $0$. 
In fact, writing $\La_n^{(d)}$ for symmetric polynomials in $n$ variables of total degree $d$, one has 
\[
\cdots \to \La_{n+1}^{(d)} \to \La_n^{(d)} \to \La_{n-1}^{(d)} \to \cdots \to 0
\]
with the same maps. The inverse limit $\La^{(d)}$ of these systems may be viewed as symmetric polynomials of degree $d$ in infinitely many variables. From the ring structure on each $\La_n$ one gets a natural ring structure on $\La := \bigoplus_{d \geq 0} \La^{(d)}$, and we call this the \emph{ring of symmetric functions}. Because $p_k(x_1,\ldots,x_{n+1}) \mapsto p_k(x_1,\ldots,x_n)$ and $m_\la(x_1,\ldots,x_{n+1}) \mapsto m_\la(x_1,\ldots,x_n)$ (for $n \geq \len(\la)$) under the natural map $\La_{n+1} \to \La_n$, these families of symmetric polynomials define symmetric functions $p_k, m_\la \in \La$. An equivalent definition of $\La$ is $\Lambda := \C[p_1,p_2,\ldots]$ where $p_i$ are indeterminates; under the natural map $\Lambda \to \Lambda_n$ one has $p_i \mapsto p_i(x_1,\ldots,x_n)$. We will often denote symmetric functions such as $p_\la,m_\la$ by $p_\la(\bx),m_\la(\bx)$, where the notation $\bx$ is meant to evoke the variables $x_1,\ldots,x_n$ in the specialization $\La \to \La_n$.

\subsection{Macdonald symmetric functions and polynomials.} Another special basis for $\La$ is given by the \emph{Macdonald symmetric functions} $P_\la(\bx;q,t)$, which depend on two additional parameters $q$ and $t$ which may be treated as indeterminates or complex numbers, though in probabilistic contexts we take $q,t \in (-1,1)$. Our first definition of them requires a certain scalar product on $\La$.

\begin{defi}
We define a scalar product $\lan \cdot, \cdot \ran_{q,t}$ on $\La$ by specifying its values on the basis $p_\la(\bx)$, which are 
\begin{equation}
\lan p_\la, p_\mu \ran_{q,t} := \bbone(\la=\mu) \prod_{i=1}^{\len(\la)} \frac{1-q^{\la_i}}{1-t^{\la_i}} \prod_{i \geq 1} i^{m_i(\la)}(m_i(\la))!,
\end{equation}
where $\bbone(\cdots)$ is the indicator function.
\end{defi}

\begin{defi}\label{def:mac_poly_torus}
The \emph{Macdonald symmetric functions} $P_\la(\bx;q,t), \la \in \Y$ are defined by the following two properties:
\begin{enumerate}
\item They are `monic' and upper-triangular with respect to the $m_\la(\bx)$ basis, in the sense that they expand as 
\begin{equation}
P_\la(\bx;q,t) = m_\la(\bx) + \sum_{\mu < \la} R_{\la \mu}(q,t) m_\mu(\bx)
\end{equation}
for some coefficients $R_{\la,\mu}(q,t)$, where $<$ denotes the lexicographic order.
\item They are orthogonal with respect to $\lan \cdot, \cdot \ran_{q,t}$. 
\end{enumerate}
\end{defi}

These conditions \emph{a priori} overdetermine the set $\{P_\la(\bx;q,t): \la \in \Y_n\}$, and it is a theorem \cite[VI (4.7)]{mac} that the Macdonald symmetric polynomials do indeed exist. It is then also clear that they form a basis for $\La$, since the $m_\la(\bx)$ do. The images of $P_\la(\bx;q,t)$ under the natural map $\La \to \La_n$ are called Macdonald polynomials, and denoted $P_\la(x_1,\ldots,x_n;q,t)$.

\begin{defi}
Define the dual Macdonald functions
\begin{equation}
Q_\la(\bx;q,t) := \frac{P_\la(\bx;q,t)}{\lan P_\la, P_\la \ran_{q,t}}.
\end{equation}
\end{defi}

Because the $P_\la$ form a basis for the vector space of symmetric polynomials in $n$ variables, there exist symmetric polynomials $P_{\la/\mu}(x_1,\ldots,x_{n-k};q,t) \in \La_{n-k}$ indexed by $\la \in \Y_{n+k}, \mu \in \Y_n$ which are defined by
\begin{equation}\label{eq:def_skewP}
    P_\la(x_1,\ldots,x_{n+k};q,t) = \sum_{\mu \in \Y_n} P_{\la/\mu}(x_{n+1},\ldots,x_{n+k};q,t) P_\mu(x_1,\ldots,x_n;q,t),
\end{equation}
often called the \emph{branching rule}. It follows easily from \eqref{eq:def_skewP} that for any $1 \leq d \leq k-1$,
\begin{equation}\label{eq:gen_branch}
    P_{\la/\mu}(x_1,\ldots,x_k;q,t) = \sum_{\nu \in \Y_{n+d}} P_{\la/\nu}(x_{d+1},\ldots,x_k;q,t) P_{\nu/\mu}(x_1,\ldots,x_d;q,t).
\end{equation}
We define $Q_{\la/\mu}$ by \eqref{eq:def_skewP} with $Q$ in place of $P$, and it is similarly clear that \eqref{eq:gen_branch} holds for $Q$. For any $\la,\nu \in \Y$, 
\begin{equation}\label{eq:skewP_consistent}
P_{\la/\mu}(x_1,\ldots,x_k,0;q,t) = P_{\la/\mu}(x_1,\ldots,x_k;q,t)
\end{equation}
for all large enough $k$, and similarly for $Q_{\la/\mu}$, hence there are symmetric functions $P_{\la/\mu}(\bx;q,t)\in \La$ and $Q_{\la/\mu}(\bx;q,t) \in \La$ which map to $P_{\la/\mu}(x_1,\ldots,x_k;q,t),Q_{\la/\mu}(x_1,\ldots,x_k;q,t)$ under the maps $\La \to \La_k$. See \cite{mac} for details.

\subsection{Specializations.} Given a symmetric polynomial $f \in \La_n$, one may substitute scalars for the variables $x_1,\ldots,x_n$ to obtain a complex number. We would like to imitate this in the setting of symmetric functions, i.e. we would like homomorphisms $\phi: \La \to \C$, which are called \emph{specializations} of $\La$. One such class of specializations is given by composing one of the projection maps $\La \to \La_n$ with the map $\La_n \to \C$ given by substituting complex numbers in for the variables, but there are others.

In probabilistic settings, one additionally often wants such homomorphisms to take nonnegative values on the Macdonald polynomials $P_\la,Q_\la$, called \emph{Macdonald-nonnegative specializations}. A full classification of these was conjectured by Kerov \cite{kerov1992generalized} and proven by Matveev \cite{matveev2019macdonald}. We describe them now: they are associated to triples of $\{\alpha_n\}_{n \geq 1}, \{\beta_n\}_{n \geq 1},\tau$ such that $\tau \in \R_{\geq 0}$, $0 \leq \alpha_n,\beta_n$ for all $ n \geq 1$, and $\sum_n \alpha_n, \sum_n \beta_n < \infty$. These are typically called usual (or alpha) parameters, dual (or beta) parameters, and the Plancherel parameter respectively. Given such a triple, the corresponding specialization is defined by 
\begin{align}\label{eq:p_specs}
\begin{split}
    p_1 &\mapsto \sum_{n \geq 1} \alpha_n + \frac{1-q}{1-t}\left(\tau + \sum_{n \geq 1} \beta_n\right) \\
    p_k &\mapsto \sum_{n \geq 1} \alpha_n^k + (-1)^{k-1}\frac{1-q^k}{1-t^k}\sum_{n \geq 1} \beta_n^k \quad \quad \text{ for all }k \geq 2.
\end{split}
\end{align}

\begin{rmk}\label{rmk:sub_scalars}
If the $\beta_i$ and $\gamma$ are all $0$ and there is a finite number of nonzero parameters $\alpha_1,\ldots,\alpha_n$, then the specialization of \eqref{eq:p_specs} is given by substituting those parameters for the variables in $p_k(x_1,\ldots,x_n)$.
\end{rmk}

For arbitrary complex numbers $\alpha_n,\beta_n$ and $\tau$ satisfying convergence conditions, \eqref{eq:p_specs} defines a specialization $\La \to \C$, but it will not in general be nonnegative. We will often consider such specializations, and when we give a triple of parameters $\{\alpha_n\}_{n \geq 1}, \{\beta_n\}_{n \geq 1},\tau$, we will not assume they define a Macdonald-nonnegative specialization unless specifically stated. For instance, in the following definition we do not assume the parameters define a Macdonald-nonnegative specialization.

\begin{defi}\label{def:gen_spec}
For the specialization $\theta:\La \to \C$ defined by the triple $\{\alpha_n\}_{n \geq 1}, \{\beta_n\}_{n \geq 1}, \tau$, and $f(\bx) \in \La$, we write
\begin{equation}\label{eq:spec_argument_notation}
f(\alpha(\alpha_1,\alpha_2,\ldots),\beta(\beta_1,\beta_2,\ldots),\gamma(\tau)) := f(\theta) := \theta(f)
\end{equation}
We will omit the $\alpha(0,0,\ldots)$ in notation if all alpha parameters are zero for the given specialization, and similarly for $\beta(0,0,\ldots)$ and $\gamma(0)$. Also, given two specializations $\theta,\theta'$, we will write
\begin{equation}\label{eq:sum_spec}
 f(\theta,\theta') := \theta(f)+\theta'(f).
 \end{equation}
\end{defi}

Note that for the specializations $\theta,\theta'$ as in \Cref{def:gen_spec}, the specialization of \eqref{eq:sum_spec} is also of the form of \Cref{def:gen_spec}; its (multi-)set of $\alpha$ parameters is just the multiset union of the alpha parameters of $\theta$ and of $\phi$, similarly for $\beta$ parameters, and its Plancherel parameter is just $\tau+\tau'$ where $\tau,\tau'$ are the Plancherel parameters of $\theta,\theta'$. 

We refer to a specialization of the form \eqref{eq:p_specs} as \emph{pure alpha} if $\tau$ and all $\beta_n, n \geq 1$ are $0$, and \emph{pure beta} if $\tau$ and all $\alpha_n, n \geq 1$ are $0$, and \emph{Plancherel} if only $\tau$ is nonzero. Note that the notations \eqref{eq:spec_argument_notation} and \eqref{eq:sum_spec} are consistent, as the specialization defined by the triple $\{\alpha_n\}_{n \geq 1}, \{\beta_n\}_{n \geq 1}, \tau$ is a sum of three specializations, one pure alpha, one pure beta, one Plancherel. On Macdonald symmetric functions, the pure alpha and pure beta specializations act as follows. 

\begin{prop}\label{thm:specialize_mac_poly}
Let $\la,\mu \in Y$ and $c_1,\ldots,c_n \in \C$. Then
\begin{align}\label{eq:spec_mac_pol}
\begin{split}
P_{\la/\mu}(\alpha(c_1,\ldots,c_n);q,t) &= P_{\la/\mu}(c_1,\ldots,c_n;q,t) \\ 
Q_{\la/\mu}(\alpha(c_1,\ldots,c_n);q,t) &= Q_{\la/\mu}(c_1,\ldots,c_n;q,t) \\ 
P_{\la/\mu}(\beta(c_1,\ldots,c_n);q,t) &= Q_{\la'/\mu'}(c_1,\ldots,c_n;t,q) \\ 
Q_{\la/\mu}(\beta(c_1,\ldots,c_n);q,t) &= P_{\la'/\mu'}(c_1,\ldots,c_n;t,q),
\end{split}
\end{align}
where in each case the left hand side is a specialized skew Macdonald symmetric function while the right hand side is a skew Macdonald polynomial with complex numbers plugged in for the variables.
\end{prop}

The alpha case of \eqref{eq:spec_mac_pol} is clear from \eqref{eq:p_specs} in view of \Cref{rmk:sub_scalars}. The beta case follows from properties of a certain involution on $\La$, see \cite[Chapter VI]{mac}, and shows that the beta parameters should still be viewed as substituting variables into a symmetric polynomial, after applying this involution. Note that the parameters $q,t$ are interchanged in the beta version, lines $3$ and $4$ of \eqref{eq:spec_mac_pol}.

The next result is in some sense trivial once notation is set up, and is also what underlies our proof of \Cref{thm:invert_moments_no_level}.

\begin{prop}\label{thm:specs_cancel}
Let $u \in \C$. Then any $\la,\mu \in \Y$,
\begin{align}\label{eq:branch_ind}
\sum_{\nu: \mu \subset \nu \subset \la} P_{\la/\nu}(\alpha(u,ut,\ldots);q,t)Q_{\nu'/\mu'}(\alpha(-u,-uq,\ldots);t,q) =  \bbone(\la=\mu).
\end{align}
\end{prop}
\begin{proof}
Let $\theta: \La \to \C$ be the specialization with parameters $\alpha(u,ut,\ldots),\beta(-u,-uq,\ldots),\gamma(0)$. We claim that both sides of \eqref{eq:branch_ind} are equal to $P_{\la/\mu}(\alpha(u,ut,\ldots),\beta(-u,-uq,\ldots);q,t) = P_{\la/\mu}(\theta;q,t)$. For the equality with the left hand side of \eqref{eq:branch_ind}, we first note that the extension of the branching rule \eqref{eq:gen_branch} to specializations holds, i.e.
\begin{equation}
P_{\la/\mu}(\phi,\phi';q,t) = \sum_{\nu: \mu \subset \nu \subset \la} P_{\la/\nu}(\phi;q,t)P_{\nu/\mu}(\phi';q,t),
\end{equation}
see e.g. \cite[(2.24)]{borodin2014macdonald}. Applying this with $\phi = \alpha(u,ut,\ldots),\phi'=\beta(-u,-uq,\ldots)$, and then noting that
\begin{equation}
P_{\nu/\mu}(\beta(-u,-uq,\ldots);q,t) = Q_{\nu'/\mu'}(\alpha(-u,-uq,\ldots);t,q)
\end{equation} 
by taking a limit of \Cref{thm:specialize_mac_poly} as $n \to \infty$, yields the desired equality. 

By \eqref{eq:p_specs}, $\theta(p_k) = 0$ for all $k \geq 1$. If $\la=\mu$ then $P_{\la/\mu}(\bx;q,t)=1 \in \La$ since $P_{\la/\mu}(x_1,\ldots,x_n;q,t) = 1$ by \eqref{eq:def_skewP}. Otherwise $P_{\la/\mu}$ is a polynomial in the $p_i$'s with no constant term and hence $\theta(P_{\la/\mu})=0$. This shows the equality with the right hand side of \eqref{eq:branch_ind}. 
\end{proof}

\begin{rmk}\label{rmk:about_polys}
We note that while the proof of \Cref{thm:specs_cancel} passes through specializations of the algebra of symmetric functions, the statement itself is essentially about symmetric polynomials and does not require the notion of specialization. We say `essentially' because $P_{\la/\nu}(\alpha(u,ut,\ldots);q,t)$ is a limit of the polynomials $P_{\la/\nu}(u,ut,\ldots,ut^{n-1};q,t)$ as $n \to \infty$, rather than a bona fide polynomial, and similarly for $Q_{\nu'/\mu'}(\alpha(-u,-uq,\ldots);t,q)$.
\end{rmk}

\subsection{Formulas in the Hall-Littlewood and $q$-Whittaker cases.} In this work we will only need two special cases of Macdonald functions/polynomials: the Hall-Littlewood case $q=0$, and the $q$-Whittaker case $t=0$. For a quick summary of relations between Hall-Littlewood polynomials abelian $p$-groups, see \cite[Sections 5 and 6]{nguyen2022universality}, and for a longer account see \cite[Chapters II and III]{mac}.

\begin{rmk}\label{rmk:qw_and_karp}
For us the $q$-Whittaker polynomials appear combinatorially via the dual specialization in \Cref{thm:specs_cancel}, but they have also been linked to linear algebra over finite fields in Karp-Thomas \cite{karp2022q}, and this should be related to their appearance here.
\end{rmk}

The pure alpha specialization $\alpha(u,ut,\ldots,ut^{n-1})$, often referred as a \emph{principal specialization}, produces simple factorized expressions for Macdonald polynomials. For brevity we give only the Hall-Littlewood case which is needed later. The formula below follows directly from \cite[Ch. III.2, Ex. 1]{mac}.

\begin{prop}[Principal specialization formula]\label{thm:hl_principal_formulas}
For $\la \in \Y_n$,
\begin{align}
\begin{split}
    P_\la(u,ut,\ldots,ut^{n-1};0,t) &= u^{|\la|} t^{n(\la)} \frac{(t;t)_n}{(t;t)_{n-\len(\la)}\prod_{i \geq 1} (t;t)_{m_i(\la)}} \\ 
    Q_\la(u,ut,\ldots,ut^{n-1};0,t) &= u^{|\la|} t^{n(\la)} \frac{(t;t)_n}{(t;t)_{n-\len(\la)}}
\end{split}
\end{align}
Similarly, for $\la \in \Y$,
\begin{align}
\begin{split}
P_\la(\alpha(u,ut,\ldots);0,t) &= u^{|\la|} t^{n(\la)} \frac{1 }{\prod_{i \geq 1} (t;t)_{m_i(\la)}} \\ 
Q_\la(\alpha(u,ut,\ldots);0,t) &= u^{|\la|} t^{n(\la)}.
\end{split}
\end{align}
Here $n(\la)$ is as in \eqref{eq:intro_notation}.
\end{prop}

For the next formula we use the \emph{$q$-binomial coefficient}
\begin{equation}
\sqbinom{a}{b}_q := \frac{(q;q)_a}{(q;q)_b(q;q)_{a-b}}
\end{equation}
for $0 \leq b \leq a$, where we recall the $q$-Pochhammer symbol \eqref{eq:def_qp}. 

\begin{lemma}[{$q$-Whittaker branching rule}]\label{thm:qw_branch_formulas}
Let $\la,\mu \in \Y$ with $\mu \prec \la$. Then
\begin{align}\label{eq:qw_branch}
\begin{split}
P_{\la/\mu}(x;q,0) &= x^{|\la|-|\mu|} \prod_{i=1}^{\len(\mu)} \sqbinom{\la_i-\la_{i+1}}{\la_i-\mu_i}_q\\ 
Q_{\la/\mu}(x;q,0) &= x^{|\la|-|\mu|} \frac{1}{(q;q)_{\la_1-\mu_1} }\prod_{i=1}^{\len(\la)-1} \sqbinom{\mu_i-\mu_{i+1}}{\mu_i-\la_{i+1}}_q.
\end{split}
\end{align}
If it is not true that $\mu \prec \la$ then all of the above are $0$.
\end{lemma}

\section{Proofs and extensions}\label{sec:moment_method}

We recall that given two finite groups $G,H$, we write $\Sur(G,H)$ for the set of surjective group homomorphisms $\phi:G \to H$. We require a result relating the number of such surjections to the Hall-Littlewood polynomials.

\begin{prop}[{\cite[Proposition 6.2]{nguyen2022universality}}]\label{thm:hl_sur}
Fix a prime $p$, let $t=1/p$, let $\la,\mu \in \Y$, and let $G_\la,G_\mu$ be as in \eqref{eq:intro_notation}. Then
\begin{equation}
\#\Sur(G_\la,G_\mu) = \frac{P_{\la/\mu}(t,t^2,\ldots;0,t)}{P_\la(t,t^2,\ldots;0,t)Q_\mu(1,t,\ldots;0,t)}.
\end{equation}
\end{prop}

We now prove the main result stated in the Introduction.

\begin{proof}[Proof of \Cref{thm:invert_moments_no_level}]
By \Cref{thm:hl_sur},
\begin{equation}
M_{G_\mu} := \E[\#\Sur(G,G_\mu)] = \sum_{\la \in \Y} \Pr(G \cong G_\la)\frac{P_{\la/\mu}(t,t^2,\ldots;0,t)}{P_\la(t,t^2,\ldots;0,t)Q_\mu(1,t,\ldots;0,t)}. 
\end{equation}
Specializing \Cref{thm:specs_cancel} to the Hall-Littlewood case and $u=t$ yields
\begin{equation}\label{eq:cancel_for_moments}
\sum_{\substack{\mu \in \Y \\ \mu \supset \nu}} P_{\la/\mu}(t,t^2,\ldots;0,t) Q_{\mu'/\nu'}(-t;t,0) = \bbone(\la=\nu).
\end{equation}
Hence 
\begin{align}\label{eq:prob_from_moments}
\begin{split}
&\sum_{\substack{\mu \in \Y \\ \mu \supset \nu}} Q_{\mu'/\nu'}(-t;t,0)Q_\mu(1,t,\ldots;0,t)\E[\#\Sur(G,G_\mu)] \\
&= \sum_{\substack{\mu \in \Y \\ \mu \supset \nu}} \sum_{\la \in \Y} \Pr(G \cong G_\la) \frac{P_{\la/\mu}(t,t^2,\ldots;0,t)}{P_\la(t,t^2,\ldots;0,t)}Q_{\mu'/\nu'}(-t;t,0) \\ 
&= \sum_{\la \in \Y} \frac{\Pr(G \cong G_\la) }{P_\la(t,t^2,\ldots;0,t)} \sum_{\substack{\mu \in \Y \\ \mu \supset \nu}}  P_{\la/\mu}(t,t^2,\ldots;0,t) Q_{\mu'/\nu'}(-t;t,0) \\ 
&= \sum_{\la \in \Y} \frac{\Pr(G \cong G_\la) }{P_\la(t,t^2,\ldots;0,t)}\bbone(\la=\nu) \\ 
&= \frac{\Pr(G \cong G_\nu)}{P_\nu(t,t^2,\ldots;0,t)}
\end{split}
\end{align}
where we used \eqref{eq:cancel_for_moments}, and the sums commute because the one in \eqref{eq:invert_moments_no_level} is absolutely convergent and the sum defining $M_{G_\mu}$ has all terms nonnegative, so the double sum is absolutely convergent.

By \Cref{thm:hl_principal_formulas} and \Cref{thm:qw_branch_formulas}
\begin{equation}\label{eq:compute_hl_product}
P_\nu(t,t^2,\ldots;0,t)Q_{\mu'/\nu'}(-t;t,0)Q_\mu(1,t,\ldots;0,t) = \frac{(-1)^{|\mu|-|\nu|} t^{n(\nu)+n(\mu)+|\mu|}}{\prod_{i \geq 1} (t;t)_{\mu_i'-\nu_i'}(t;t)_{\nu_i'-\mu_{i+1}'}}
\end{equation}
when $\mu' \succ \nu'$, and is $0$ otherwise. Combining \eqref{eq:prob_from_moments} with \eqref{eq:compute_hl_product} shows \eqref{eq:invert_moments_no_level}. Uniqueness is clear from \eqref{eq:invert_moments_no_level}.
\end{proof}

Given a random abelian group $G$, it is often desirable to study the quotient $G/p^dG$. For instance, in \cite{wood2017distribution} and subsequent works on cokernels of random integer matrices, one always reduces the cokernel modulo some integer $a$, for then one only has to study random matrices over a finite ring $\Z/a\Z$. \Cref{thm:invert_moments_no_level} admits the following refinement, which concerns the $p^d$-torsion group $G/p^dG$ and consequently allows $G$ to be potentially infinite (though its $p$-part must be finitely generated for the moment finiteness condition to be satisfied). In what follows we let $\Y_d' = \{\la \in \Y: \len(\la') \leq d\}$.

\begin{thm}\label{thm:invert_moments_fixed_level}
Fix $p$ prime and $d \in \N$, and let $G$ be a random abelian group with finite $H$-moments $M_H := \E[\#\Sur(G,H)]$ for each finite abelian $p^d$-torsion group $H$. Then for any $\nu \in \Y_d'$,
\begin{equation}\label{eq:invert_moments_fixed_level}
\Pr(G/p^d G \cong G_\nu) = \sum_{\substack{\mu \in \Y_d' \\ \mu' \succ \nu'}} \frac{(-1)^{|\mu|-|\nu|} t^{n(\nu)+n(\mu)+|\mu|}}{\prod_{i \geq 1} (t;t)_{\mu_i'-\nu_i'}(t;t)_{\nu_i'-\mu_{i+1}'}} M_{G_\mu},
\end{equation}
where $t=1/p$, provided that the sum on the right hand side converges absolutely.
\end{thm}
\begin{proof}
Simply note that if $G$ in \Cref{thm:invert_moments_no_level} is $p^d$-torsion, then the $H$-moments $M_H$ are only nonzero if $H$ is $p^d$-torsion, so the summands on the right hand side of \eqref{eq:invert_moments_no_level} are only nonzero if $\mu \in \Y_d'$.
\end{proof}

\begin{rmk}
Note that the condition $\mu' \succ \nu'$ in the sum implies that of the nontrivial conjugate parts $\mu_1',\ldots,\mu_d'$, all except $\mu_1'$ are bounded above and below.
\end{rmk}

\begin{rmk}
One may also show \Cref{thm:invert_moments_fixed_level} by noting that for $p^d$-torsion $H$, $\#\Sur(G,H)$ depends only on $G/p^dG$, and then doing similar manipulations to \eqref{eq:prob_from_moments} with a sum over $\mu_d'$. The dependence of the moments on $G/p^dG$ may be seen directly in symmetric function theory (though this would be a very convoluted way to arrive at it): By the explicit formulas \Cref{thm:hl_principal_formulas} and their generalization \cite[Theorem 3.3]{vanpeski2021halllittlewood} to skew Hall-Littlewood functions, for any $\lambda \in \Y, \mu \in \Y_d'$ we have
\begin{equation}
\label{eq:hl_quotient}
\frac{P_{\lambda/\mu}(1,t,\ldots; 0,t)}{P_\lambda(1,t,\ldots;0,t)} = \prod_{i=1}^d t^{\binom{\lambda_i' - \mu_i'}{2} - \binom{\lambda_i'}{2}} (t^{1+\lambda_i'-\mu_i'};t)_{\mu_i' - \mu_{i+1}'},
\end{equation}
which is independent of $\la_{d+1}',\la_{d+2}',\ldots$. In the special case of \Cref{thm:invert_moments_fixed_level} given in \cite[Lemma 4.2]{van2023local}, we were not yet aware of the simpler proof of \Cref{thm:invert_moments_fixed_level} via \Cref{thm:invert_moments_no_level} and basic group theory given above, and so argued directly in this way similarly to \eqref{eq:prob_from_moments} using \eqref{eq:hl_quotient}.
\end{rmk}

\subsection{Multiple primes.} We now extend \Cref{thm:invert_moments_no_level} to multiple primes. Because we are considering more than one prime, we change notation slightly and write $G_{\la,p} := \bigoplus_{i \geq 1} \Z/p^{\la_i}\Z$ for partitions $\la \in \Y$ (which before we denoted by $G_\la$ with $p$ implicit).

\begin{defi}
For any finite collection $P=\{p_1,\ldots,p_k\}$ of primes, we let $\cA_P$ the set of (isomorphism classes of) finite abelian groups $H$ such that every prime factor of $\#H$ lies in $P$. We further define, for any partitions $\la(1),\ldots,\la(k) \in \Y$, the group
\begin{equation}
G_{\la(1),\ldots,\la(k)}(P) := \bigoplus_{i=1}^k G_{\la(i),p_i}.
\end{equation}
\end{defi}

It is clear that $\cA_P = \{G_{\la(1),\ldots,\la(|P|)}(P): (\la(1),\ldots,\la(|P|)) \in \Y^{|P|}\}$.

\begin{thm}\label{thm:multiple_primes}
Let $P = \{p_1,\ldots,p_k\}$ be a set of primes and $G$ be a random group in $\cA_P$ with finite $H$-moments $M_H := \E[\#\Sur(G,H)]$ for each $H \in \cA_P$. Then for any $(\nu(1),\ldots,\nu(k)) \in \Y^k$,
\begin{multline}
\label{eq:invert_moments_multiple_primes}
\Pr(G \cong G_{\nu(1),\ldots,\nu(k)}(P)) \\ 
= \sum_{\substack{(\mu(1),\ldots,\mu(k)) \in \Y^k \\ \mu(i)' \succ \nu(i)' \text{ for all }i}} \prod_{i=1}^k \frac{(-1)^{|\mu(i)|-|\nu(i)|} p_i^{-n(\nu(i))-n(\mu(i))-|\mu(i)|}}{\prod_{j \geq 1} (p_i^{-1};p_i^{-1})_{\mu(i)_j'-\nu(i)_j'}(p_i^{-1};p_i^{-1})_{\nu(i)_j'-\mu(i)_{j+1}'}} M_{G_{\mu(1),\ldots,\mu(k)}(P)},
\end{multline}
provided the sum on the right hand side converges absolutely. Furthermore, the law of $G$ is the unique probability measure on $\cA_P$ with moments $M_H$.
\end{thm}
\begin{proof}
It suffices to show the case of constant $G$, i.e. the equality 
\begin{multline}
\label{eq:invert_moments_multiple_primes_equivalent}
\bbone(G \cong G_{\nu(1),\ldots,\nu(k)}(P)) \\ = \sum_{\substack{(\mu(1),\ldots,\mu(k)) \in \Y^k \\ \mu(i)' \succ \nu(i)' \text{ for all }i}} \prod_{i=1}^k \frac{(-1)^{|\mu(i)|-|\nu(i)|} p_i^{-n(\nu(i))-n(\mu(i))-|\mu(i)|}}{\prod_{j \geq 1} (p_i^{-1};p_i^{-1})_{\mu(i)_j'-\nu(i)_j'}(p_i^{-1};p_i^{-1})_{\nu(i)_j'-\mu(i)_{j+1}'}} \#\Sur(G,G_{\mu(1),\ldots,\mu(k)}(P))
\end{multline}
of functions on $\cA_P$, as the general case may be obtained by linear combinations of the above. We do this by induction on $k$, the base case $k=1$ following from \Cref{thm:invert_moments_no_level}. Suppose \eqref{eq:invert_moments_multiple_primes_equivalent} holds for some $k-1$ and let $p_1,\ldots,p_{k}$ be distinct primes. Writing 
\begin{equation}
\bbone(G \cong G_{\nu(1),\ldots,\nu(k)}(P)) = \bbone\left(\bigoplus_{i=1}^{k-1} G_{p_i} \cong \bigoplus_{i=1}^{k-1} G_{\nu(i),p_i}\right) \cdot \bbone(G_{p_k} \cong G_{\nu(k),p_k}),
\end{equation}
we may apply the inductive hypothesis to each factor, and the conclusion follows after noting that 
\begin{equation}
\#\Sur(G,G_{\mu(1),\ldots,\mu(k)}(P)) = \#\Sur\left(\bigoplus_{i=1}^{k-1} G_{p_i},\bigoplus_{i=1}^{k-1} G_{\mu(i),p_i}\right) \cdot \#\Sur(G_{p_k},G_{\mu(k),p_k}).
\end{equation}
Uniqueness is immediate as in \Cref{thm:invert_moments_no_level}.
\end{proof}

\begin{rmk}
The same proof as for \Cref{thm:invert_moments_fixed_level} yields a version of the above formula for $\Pr(G/(p_1^{e_1} \cdots p_k^{e_k}) \cong H)$ in the setting of \Cref{thm:multiple_primes}. It is this version for modules over the finite quotient ring $\Z/(p_1^{e_1} \cdots p_k^{e_k})$ which corresponds to the formula in \cite[Proposition 6.2]{sawin2022moment}, though one must do some calculation with that formula to reduce it to ours.
\end{rmk}

\begin{rmk}\label{rmk:dvr}
If $R$ is a complete discrete valuation ring with maximal ideal $(\omega)$ generated by a uniformizer $\omega$, and finite residue field $R/(\omega) \cong \F_q$, then \Cref{thm:invert_moments_no_level} and \Cref{thm:invert_moments_fixed_level} hold with $\Sur(G,H)$ replaced by surjective $R$-module homomorphisms and $t=p^{-1}$ replaced by $t=q^{-1}$. One only needs that \Cref{thm:hl_sur} holds in this setting, which follows because the number of surjections is known to be a rational function of $q^{-1}$ in this setting, see for instance \cite[Chapter II]{mac}.
\end{rmk}

\bibliographystyle{alpha_abbrvsort}
\bibliography{references.bib}

\end{document}

%% file: preamble.tex
\usepackage{anysize}
\usepackage{amsmath}
\usepackage{amsthm}
\usepackage{amsmath,amscd}
\usepackage[utf8]{inputenc}
\usepackage{amssymb}
\usepackage{stmaryrd}
\usepackage{wasysym}
\usepackage{mathrsfs}
\usepackage[pagebackref=true]{hyperref} 
\renewcommand*{\backref}[1]{}
\renewcommand*{\backrefalt}[4]{({\tiny%
   \ifcase #1 Not cited.%
         \or Cited on page~#2.%
         \else Cited on pages #2.%
   \fi%
   })}
\usepackage{graphicx}
\usepackage{subcaption}
\usepackage{cleveref}
\usepackage{enumitem}
\usepackage{relsize} 
\usepackage{dsfont}
\usepackage{soul}
\usepackage{filecontents}
\hypersetup{colorlinks=true,linkcolor=black,citecolor=black}
\usepackage{color}
\usepackage[all]{xy}
\usepackage{float}
\usepackage{bm}
\usepackage{mathtools}
\usepackage{thmtools}
\usepackage{setspace}
\usepackage{comment}
\usepackage{tikz}
\usepackage{ytableau}
\usepackage{mathdots}
\usepackage[myheadings]{fullpage}

\numberwithin{equation}{section}
\newcommand\mtop{1in}
\newcommand\mbottom{1in}
\newcommand\mleft{1in}
\newcommand\mright{1in}
\usepackage[top = \mtop, bottom = \mbottom, left = \mleft, right=\mright]{geometry}

\newtheorem{thm}{Theorem}[section]

\newtheorem{prop}[thm]{Proposition}

\newtheorem{lemma}[thm]{Lemma}

\theoremstyle{definition}
\newtheorem{defi}{Definition}

\newtheorem{rmk}{Remark}

\usepackage{scalerel,stackengine}
\stackMath
\newcommand\reallywidehat[1]{%
\savestack{\tmpbox}{\stretchto{%
  \scaleto{%
    \scalerel*[\widthof{\ensuremath{#1}}]{\kern-.6pt\bigwedge\kern-.6pt}%
    {\rule[-\textheight/2]{1ex}{\textheight}}
  }{\textheight}%
}{0.5ex}}%
\stackon[1pt]{#1}{\tmpbox}%
}
\DeclareSymbolFont{bbold}{U}{bbold}{m}{n}
\DeclareSymbolFontAlphabet{\mathbbold}{bbold}

\makeatletter
\def\@tocline#1#2#3#4#5#6#7{\relax
  \ifnum #1>\c@tocdepth 
  \else
    \par \addpenalty\@secpenalty\addvspace{#2}%
    \begingroup \hyphenpenalty\@M
    \@ifempty{#4}{%
      \@tempdima\csname r@tocindent\number#1\endcsname\relax
    }{%
      \@tempdima#4\relax
    }%
    \parindent\z@ \leftskip#3\relax \advance\leftskip\@tempdima\relax
    \rightskip\@pnumwidth plus4em \parfillskip-\@pnumwidth
    #5\leavevmode\hskip-\@tempdima
      \ifcase #1
       \or\or \hskip 1em \or \hskip 2em \else \hskip 3em \fi%
      #6\nobreak\relax
    \hfill\hbox to\@pnumwidth{\@tocpagenum{#7}}\par
    \nobreak
    \endgroup
  \fi}
\makeatother

\makeatletter
\newcommand{\subalign}[1]{%
  \vcenter{%
    \Let@ \restore@math@cr \default@tag
    \baselineskip\fontdimen10 \scriptfont\tw@
    \advance\baselineskip\fontdimen12 \scriptfont\tw@
    \lineskip\thr@@\fontdimen8 \scriptfont\thr@@
    \lineskiplimit\lineskip
    \ialign{\hfil$\m@th\scriptstyle##$&$\m@th\scriptstyle{}##$\hfil\crcr
      #1\crcr
    }%
  }%
}
\makeatother

\newcommand{\R}{\mathbb{R}}
\newcommand{\Z}{\mathbb{Z}}

\newcommand{\N}{\mathbb{N}}
\newcommand{\C}{\mathbb{C}}
\newcommand{\F}{\mathbb{F}}

\newcommand{\E}{\mathbb{E}}

\newcommand{\bbone}{\mathbbold{1}}
\newcommand{\la}{\lambda}
\newcommand{\La}{\Lambda}

\newcommand{\lan}{\left\langle}
\newcommand{\ran}{\right\rangle}

\newcommand{\Y}{\mathbb{Y}}

\newcommand{\bx}{\mathbf{x}}

\newcommand{\cP}{\mathbb{Y}}

\DeclareMathOperator{\len}{len}

\newcommand{\sqbinom}[2]{\begin{bmatrix}#1\\ #2\end{bmatrix}}

\newcommand{\Sur}{\operatorname{Sur}}

\newcommand{\cA}{\mathcal{A}}

